\documentclass[smallextended,final]{svjour3}       % onecolumn (second format)
\usepackage[compatibility=false]{caption}

\usepackage{amsmath}
\usepackage{amssymb}
\usepackage{amsmath}
\usepackage{algorithm}
\usepackage{graphicx}
\usepackage[noend]{algpseudocode}
\usepackage{mathtools} %for entry alignment in matrices
\usepackage{hyperref}
\usepackage{subcaption} % to put figures side-by-side
\usepackage{multicol}

\usepackage{scalerel}   % for thinbar
\input{insbox}

\usepackage{geometry}
\DeclareCaptionType[within=section]{mat}[Polytope]

\def\CF{\mathcal{F} }
\def\CC{\mathcal{C} }
\newtheorem{open}{Question:}

\newenvironment{proofof}[1]{ \par\noindent{\it{Proof of #1.}}\ }%{\hfill$\square$\par}

\usepackage{pgffor}
\foreach \x in {A,...,Z}{%
	\expandafter\xdef\csname b\x\endcsname{\noexpand\ensuremath{\noexpand\mathbf{\x}}}
	\expandafter\xdef\csname c\x\endcsname{\noexpand\ensuremath{\noexpand\mathcal{\x}}}
	\expandafter\xdef\csname B\x\endcsname{\noexpand\ensuremath{\noexpand\mathbb{\x}}}
}

\foreach \x in {a,...,z}{%
	\expandafter\xdef\csname b\x\endcsname{\noexpand\ensuremath{\noexpand\mathbf{\x}}}
}

\newcommand{\MaxBond}{{\sc Max-Bond }}
\newcommand{\MaxCut}{{\sc Max-Cut }}
\newcommand{\MB}{\ensuremath{\mathrm{MAXB}}}
\newcommand{\bondp}[1]{\ensuremath{\mathrm{BOND}\left(#1\right)}}
\newcommand{\abondp}[2]{\ensuremath{\mathrm{BOND}\left(#1,#2\right)}}
\newcommand{\conv}[1]{\ensuremath{\mathrm{conv}\left(#1\right)}}
\newcommand{\ksum}[1]{\ensuremath{\oplus_{#1}}}
\newcommand{\xc}[1]{\ensuremath{\mathrm{xc}\left(#1\right)}}
\newcommand{\vertices}[1]{\ensuremath{\mathrm{vert}\left(#1\right)}}
\newcommand{\Oh}[1]{\ensuremath{\cO}\left(#1\right)}
\newcommand{\prism}{\textit{Prism}}
\newcommand{\vertexset}{\mathop{\mathrm{vert}}}

\newcommand{\osfl}{\textsf{Best-So-Far-ind-L}}
\newcommand{\absfl}{\textsf{BSFL}}

\newcommand{\osf}{\textsf{Best-So-Far}}
\newcommand{\absf}{\textsf{BSF}}
\newcommand{\osfr}{\textsf{Best-So-Far-ind-R}}
\newcommand{\absfr}{\textsf{BSFR}}

\newcommand{\po}{\textsf{Suffix-Best}}
\newcommand{\asb}{\textsf{SB}}
\newcommand{\asbl}{\textsf{SBL}}
\newcommand{\pol}{\textsf{Suffix-Best-ind-L}}
\newcommand{\oc}{\textsf{New-Best-Candidate}}

\newcommand{\pbl}{\textsf{Prefix-Best-Left}}
\newcommand{\pbr}{\textsf{Prefix-Best-Right}}
\newcommand{\pbli}{\textsf{Prefix-Best-Left-ind}}
\newcommand{\pbri}{\textsf{Prefix-Best-Right-ind}}
\newcommand{\plk}{\textsf{Prefix-Left}}
\newcommand{\prk}{\textsf{Prefix-Right}}
\newcommand{\bestsol}{\textsf{Best-Solution}}

\newcommand\utimes{\mathbin{\ooalign{$\cup$\cr%
   \hfil\raise0.42ex\hbox{$\scriptscriptstyle\times$}\hfil\cr}}}
\newcommand\bigutimes{\mathop{\ooalign{$\bigcup$\cr%
   \hfil\raise0.36ex\hbox{$\scriptscriptstyle\boldsymbol{\times}$}\hfil\cr}}}
\newcommand{\thinbar}[1]{\scaleto{\overline{\mkern-2mu#1\mkern-1mu}\mkern-1mu}{4.5pt}}

\makeatletter
\def\makeheadbox{}%
\makeatother

\begin{document}
%\linenumbers
\title{Bond Polytope under Vertex- and Edge-sums}
\author{Petr Kolman\thanks{Partially supported by grant 24-10306S of GA ČR.} \and Hans Raj Tiwary \thanks{Partially supported by the AGATE project funded  from the Horizon Europe Programme under Grant Agreement No. 101183743.}}
\institute{P. Kolman \and H. R. Tiwary \at Charles University, Faculty of Mathematics and Physics,
Department of Applied Mathematics,
Prague, Czech Republic,
\email{\{kolman,hansraj\}@kam.mff.cuni.cz}}
\authorrunning{P. Kolman, H. R. Tiwary}

\date{ }
\maketitle
%\vspace{-3cm}
%%%%%%%%%%%%%%%%%%%%%%%%%%%%%%%%%%%%%%%%%%%%%%%%%%%%%%%%%%%%%%%%%%%%%%%%%%%%%
\begin{abstract}   
A cut in a graph $G$ is called a {\em bond} if both parts of the cut induce
connected subgraphs in $G$, and the {\em bond polytope} is the convex hull of
all bonds. Computing the maximum weight bond is an NP-hard problem even for
planar graphs. However, the problem is solvable in linear time on $(K_5 \setminus
e)$-minor-free graphs, and in more general, on graphs of bounded treewidth,
essentially due to clique-sum decomposition into simpler graphs. 

We show how to obtain the bond polytope of graphs that are $1$- or $2$-sum of
graphs $G_1$ and $ G_2$ from the bond polytopes of $G_1,G_2$. Using this we show
that the extension complexity of the bond polytope of $(K_5 \setminus
e)$-minor-free graphs is linear. Prior to this work, a linear size description
of the bond polytope was known only for $3$-connected planar $(K_5 \setminus
e)$-minor-free graphs, essentially only for wheel graphs.

We also describe an elementary linear time algorithm for the \MaxBond problem on
$(K_5\setminus e)$-minor-free graphs. Prior to this work, a linear time
algorithm in this setting 
was known. 
However, the hidden constant in the big-Oh notation was large because the algorithm relies on the heavy machinery of linear time algorithms for graphs of bounded treewidth, used as a black box.

\end{abstract}

%%%%%%%%%%%%%%%%%%%%%%%%%%%%%%%%%%%%%%%%%%%%%%%%%%%%%%%%%%%%%%%%%%%%%%%%%%%%%
\section{Introduction}
%%%%%%%%%%%%%%%%%%%%%%%%%%%%%%%%%%%%%%%%%%%%%%%%%%%%%%%%%%%%%%%%%%%%%%%%%%%%%%%%%%%%%%%%%%%%%
The \MaxCut problem is a fundamental problem in computer science and is one of Karp's original 21 NP-Complete problems \cite{Karp:72}. Given a graph $G=(V,E)$ the problem asks for a subset $S\subseteq V$ of vertices such that the number of edges with exactly one endpoint in $S$ is as large as possible. However, in some applications such as image segmentation \cite{VKR:08}, forest planning and harvest scheduling~\cite{CCGVW:13}, and certain market zoning \cite{GKLSZ:19}, one 
imposes an additional condition that both components $G[S]$ and $G[V\setminus S]$ be {\em connected}. This version of \MaxCut has been studied by various authors \cite{EHKK:19,DEHKKLPSS:21,GHKPS:18,Hajiaghayietal:20,CJN:23,Chaourar:20} under different names, but following Duarte et al.~\cite{DLPSS:19} and Chimani et al.~\cite{CJN:23} we will refer to this as the \emph{bond problem.}

Formally, given a graph $G=(V,E)$ a {\em bond} in $G$ is 
a cut $(S,V\setminus S)$ such that the induced subgraphs $G[S]$ and $G[V\setminus S]$ are both connected; $S$ and 
$V\setminus S$ are two {\em sides} of the bond. 
Note that bonds of a connected graph $G$ are the minimal edge cuts of $G$. 
The \MaxBond problem seeks to find a bond $(S,V\setminus S)$ such that the number of edges between $S$ and $V\setminus S$ is maximized. 
For each bond in a graph $G$, we consider the characteristic vector of its edges; 
for simplicity, we do not distinguish between a bond, the edges in a bond, and the characteristic vector of the edges in a bond unless the meaning is not clear in the context of the discussion.

For a graph $G$ its bonds are circuits of the co-graphic matroid of $G$~\cite{Oxley:01}. Co-graphic matroids form an essential ingredient in the decomposition result for regular matroids by Seymour \cite{seymour:80}. Regular matroids are known to have polynomial extension complexity \cite{AF:22}.

In this paper, we deal with the \MaxBond problem on $(K_5\setminus e)$-minor-free graphs.
A linear time algorithm for bounded-treewidth graphs was given by Duerte et al. \cite{DLPSS:19}. More specifically, for $(K_5\setminus e)$-minor-free graphs, Chaourar~\cite{Chaourar:20} gave a quadratic time algorithm. % for finding a maximum bond.
Chimani et al. \cite{CJN:23} improved this result by giving a linear time algorithm;
the algorithm uses as the black box a linear time algorithm of Duarte et 
al.~\cite{DLPSS:19} for the bond problem on bounded tree-width graphs,  
and this black-box is used to get the maximum bond for the wheel graphs $W_n$ in combination with a divide-and-conquer like strategy. 
Chimani et al.~\cite{CJN:23}
also gave a characterisation of the bond polytope for $3$-connected planar $(K_5\setminus e)$-minor-free graphs by 
giving a linear size set of linear inequalities defining it; by a result of Wagner \cite{Wagner:60}, this class of graphs contains only the wheel graphs $W_n$, the triangle $K_3$, and the triangular prism. 
The question of describing the bond polytope for general $(K_5\setminus e)$-minor-free graphs
was left open.

Our contributions are twofold:
\begin{enumerate}
    \item 
    We show how to obtain linear size descriptions of the bond polytope of graphs that are $k$-sum, for $k = 1, 2$, of other graphs with known linear size bond polytope.
    Using this result, we prove that the extension complexity of the bond polytope is linear for arbitrary $(K_5\setminus e)$-minor-free graphs.  
    This, in a sense, is the best one can do because -- as we note later (Lemma~\ref{lem:1sum-vertices} and the subsequent remark)
    -- the actual description of the bond polytope even for $(K_5\setminus e)$-minor-free graphs can be exponential in the size of the graph. This answers an open question posed by Chimani et al. \cite{CJN:23}.
    \item We simplify the algorithmic result for the \MaxBond problem of Chimani et al. \cite{CJN:23} by giving a simple linear time algorithm for the wheel graph, removing the need to use the tree-width machinery~\cite{EHKK:19}, which yields a linear time algorithm for
    the \MaxBond problem for all $(K_5\setminus e)$-minor-free graphs. Chimani et al. \cite{CJN:23} mention the possibility of existence of algorithms simpler than theirs so our algorithm can be seen as an answer to their question.
\end{enumerate}

It should be noted that $(K_5\setminus e)$-minor-free graphs have bounded treewidth and bonds can be represented by a formula in Monadic Second Order (MSO) logic. So both a linear time algorithm as well as a linear size extended formulation follow readily from meta results about bounded treewidth graphs: the algorithmic results follow from the work of Courcelle~\cite{Courcelle:90} and are given explicitly by Duarte et al.~\cite{DLPSS:19}, while the polyhedral results follow from the work of Kolman et al.~\cite{KolmanKT:20}. However, 
% as we noted earlier 
the magnitude of the constants in both cases is enormous~\cite{LangerRRS:14}, in contrast
to the constants in our results.

%\subsection{Other Related Results}
\paragraph{Other Related Results}
The cut polytope for clique-sums of size three was studied by Barahona \cite{BarahonaM:86} who gave an efficient algorithm and extended formulation for the cut polytope of $K_5$-minor-free graphs.

A closely related problem is the version of the \MaxCut in which only the $S$ part is required to be
connected. This version of the \MaxCut problem is NP-hard~\cite{Hajiaghayietal:20} as well.
Schieber and Vahidi~\cite{Schieberv:23} gave an $O(\log\log n)$-approximation
improving an earlier $O(\log n)$-approximation~\cite{Hajiaghayietal:20}.

In contrast to the \MaxCut problem, there is no constant-factor approximation algorithm 
for \MaxBond unless P=NP~\cite{DEHKKLPSS:21}. On the positive side, both \MaxBond and the
version of \MaxCut with one side connected are fixed-parameter tractable when parameterized 
by the size of the solution, the treewidth, and the twin-cover number~\cite{DEHKKLPSS:21}.

%%%%%%%%%%%%%%%%%%%%%%%%%%%%%%%%%%%%%%%%%%%%%%%%%%%%%%%%%%%%%%%%%%%%%%%%%%%%%
\section{Preliminaries}
%%%%%%%%%%%%%%%%%%%%%%%%%%%%%%%%%%%%%%%%%%%%%%%%%%%%%%%%%%%%%%%%%%%%%%%%%%%%%%%%%%%%%%%%%%%%%
Let $G_1$ and $G_2$ be two graphs and $U_1 \subseteq V(G_1)$ and
$U_2 \subseteq V(G_2)$ two subsets of vertices inducing a clique of the
same size, say size $k$, for some $k\geq 1$. 
A graph $G$ is a {\em clique-sum}
of $G_1$ and $G_2$ if $G$ is obtained from $G_1$ and $G_2$ by identifying $U_1$ 
and $U_2$, and possibly removing some edges from the clique. 

In this paper, we
use the clique-sums for $k=1,2$. To distinguish between the $2$-sum
that keeps the edge in the clique, and the $2$-sum that removes it, we
denote the former operation by $\ksum{2}$ and the later by $\ksum{2}^{-}$.
If we want to emphasize that $G_1\ksum{1}G_2$ is taken over a vertex $v$, we will denote it as $G_1\ksum{v}G_2$. Similarly, $G_1\ksum{e}G_2$ or $G_1\ksum{e}^-G_2$ will be 
used to mean that the $2$-sum of $G_1$ and $G_2$ is taken over the edge $e$.

For a graph $G=(V,E)$, a pair of vertices $uv$ is called a {\em non-edge} if $uv\not \in E$.
For an edge $e\in E$, by $G\setminus e$ we denote the graph $(V,E\setminus \{e\})$,
and for $e\notin E$, by $G\cup\{e\}$ we denote the graph $(V,E\cup e)$, and
we use $uv$ as an abbreviation of $\{u,v\}$.
In the case of (edge) weighted graphs, the weight of an edge $uv$ is denoted $w(u,v)$.
For a subset $S$ of vertices, $\delta(S)$ is the set of edges between $S$ and $V\setminus S$.
%The {\em length of a path} is the number of edges of the path.

A graph $H$ is a {\em minor} of a graph $G$ if $H$ can be obtained from $G$ by 
a series of vertex and edge deletions and edge contractions. A graph $G$ is
$H$-{\em minor-free} if $H$ is not a minor of $G$.

For $n\geq 3$, a {\em wheel graph} $W_n$ is a graph with a vertex set 
$V=\{0,1,\ldots,n-1\}\cup\{c\}$, for $c\not\in \{0,1,\ldots,n-1\}$,
and an edge set $E=\bigcup_{i=0}^{n-2}\{\{i,i+1\},\{i,c\}\} \cup\{\{n-1,0\},\{n-1,c\}\}$;
the vertex $c$ is called the {\em hub} of the wheel, the cycle $0,1,\ldots,n-1,0$ is
the {\em rim}, and the edges of the form $\{i,c\}$ are the {\em spokes} of the wheel.
For integers $i<j$, let $[i,j]$ denote the set $\{i,i+1,\ldots,j\}$.
 
The {\em Prism} graph is the cartesian product of a $K_3$ with a single edge.

%%%%%%%%%%%%%%%%%%%%%%%%%5
\begin{theorem}[Satz 7, Wagner~\cite{Wagner:60}]\label{thm:wagner}
Each maximal 
$(K_5\setminus e)$-minor-free graph $G$ can be decomposed as 
$G=G_1 \oplus^1 \dots \oplus^{l-1} G_\ell$ where each
$G_i$ is isomorphic to a wheel graph, $\prism$, $K_2$, $K_3$, or $K_{3,3}$,
and each operation $\oplus^i$ is $\oplus_1$ or $\oplus_2$. % or $\oplus_2^-$.
\end{theorem}

\begin{theorem}\label{thm:decomp}
Each $(K_5\setminus e)$-minor-free graph $G$ can be decomposed in linear time as 
$G=G_1 \oplus^1 \dots \oplus^{l-1} G_\ell$ where each
$G_i$ is isomorphic to a wheel graph, $\prism$, $K_2$, $K_3$, or $K_{3,3}$,
and each operation $\oplus^i$ is $\oplus_1$, $\oplus_2$ or $\oplus_2^-$.
\end{theorem}
\begin{proofof}{Theorem~\ref{thm:decomp}}
We start by finding a decomposition of the $(K_5\setminus e)$-minor-free graph
$G$ into a tree of $2$-connected components;
this can be done in linear time by a depth-first search algorithm~\cite{HopcroftT:73b}.
The components are attached to each other at shared vertices and $G$ corresponds to
$1$-sums of these components. 

Consider now a $2$-connected component $H$ of $G$. By a linear time algorithm of
Hopcroft and Tarjan~\cite{HopcroftT:73} we construct a decomposition of $H$ 
into a tree $T$ of $3$-connected components. Informally, the nodes of $T$ are $3$-connected subgraphs of $H$ and if two nodes share an edge in $T$ then the corresponding subgraphs in $H$ share two vertices. By induction on the number of vertices
in $T$ we show that $H$ is obtained from a wheel graph, Prism, $K_2$, $K_3$, and 
$K_{3,3}$ by the operations $\oplus_2$ and $\oplus_2^-$. 

If $T$ has only a single vertex, then $H$ is a $3$-connected 
$(K_5\setminus e)$-minor-free graph; the only such graphs are a wheel graph, Prism, 
$K_2$, $K_3$, or $K_{3,3}$ (cf.~\cite{CJN:23}) which completes the proof of the base case. 

For the inductive step, assume that $T$ has at least two vertices, and let $t$ be an
arbitrary leaf of $T$. Let $H_1=(V_1,E_1)$ be the subgraph of $H$ corresponding to
$T\setminus t$, $H_2=(V_2,E_2)$ be the subgraph of $H$ corresponding to $t$,
and let $u$ and $v$ be the two vertices in $V_1\cap V_2$. We distinguish two cases.

If $uv\in E_1$, then $H=H_1\oplus_{uv}H_2$. Therefore, by our inductive hypothesis,
$H_1=G_1\oplus^1\cdots\oplus^{\ell-2}G_{\ell-1}$ where for each $i$, $G_i$ is
a wheel graph, Prism, $K_2$, $K_3$, or $K_{3,3}$, and $\oplus^i$ is either $\ksum{1},\ksum{2},$ or $\ksum{2}^-$. Since $G=H_1\ksum{2}H_2$, and $H_2$ is one of the graphs in our list, 
the proof is completed.

If $uv\notin E$, then
$H=H_1'\oplus_{uv}^-H_2'$ where $H_i'=H_i\cup \{uv\}$ for $i=1,2$. Observe that 
both $H_1'$ and $H_2'$ are
$(K_5\setminus e)$-minor-free graphs. To see this, assume without loss of generality that $H_1'$ contains a $K_5\setminus e$ minor. As $H_2'$ is a connected graph, $u$ and $v$ are connected by a path in $H_2'$ and so $H$ contains a $K_5\setminus e$ minor as well,
which is a contradiction to the fact that $G$ is $(K_5\setminus e)$-minor-free.

Therefore, by our inductive hypothesis, $H_1'=G_1\oplus^1\cdots\oplus^{\ell-2}G_{\ell-1}$ 
where for each $i$, $G_i$ is
a wheel graph, Prism, $K_2$, $K_3$, or $K_{3,3}$, and $\oplus^i$ is either $\ksum{1},\ksum{2},$ or $\ksum{2}^-$. Since $H_2'$ is a $3$-connected $(K_5\setminus e)$-minor-free graph, it is either a wheel graph, Prism, $K_2$, $K_3$, or $K_{3,3}$. Finally, as $G=H_1'\ksum{2}^-H_2'$, the proof is completed.
\qed
\end{proofof}

Let $P$ be a polytope in $\BR^d$. A polytope $Q$ in $\BR^{d+r}$ is called an
\textit{extended formulation} of $P$ if $P$ is a
projection of $Q$ onto the first $d$ coordinates.
The \textit{size} of a polytope is defined to be the number of its
facet-defining inequalities, and the \emph{extension complexity} of a polytope $P$,
denoted by $\xc{P}$, is the size of its smallest extended formulation.
%%%%%%%%%%%%%%%%%%%%%%%%%5
\begin{theorem}[Balas~\cite{Balas:98}, Theorem 2.1]\label{thm:balas}
If $P_1,\ldots, P_q$ are non-empty polytopes, then \[\xc{\conv{\bigcup_{i=1}^q P_i}}\leq q + \sum_{i=1}^q\xc{P_i}\ .\] Furthermore, such an extended formulation can be constructed from extended formulations of the $P_i$'s in linear time. 
\end{theorem}

Let $P_1\subseteq\mathbb{R}^{d_1+k}$ and $P_2\subseteq\mathbb{R}^{d_2+k}$ be two 
$0/1$-polytopes with vertices $\vertexset(P_1)$ and $\vertexset(P_2)$, respectively. The 
\emph{glued product} $P_1\times_k P_2$ of $P_1$ and $P_2$, where the gluing is done over the last $k$ 
coordinates, is defined to be
\[P_1\times_k P_2:=\mathrm{conv}\left\{\left.\begin{pmatrix}\bx\\\by\\\bz\end{pmatrix}\in
\{0,1\}^{d_1+d_2+k}\right|\begin{pmatrix}\bx\\\bz\end{pmatrix}\in\vertexset(P_1), 
\begin{pmatrix}\by\\\bz\end{pmatrix}\in\vertexset(P_2)\right\}.\]

We will use the following known result about glued products.
%%%%%%%%%%%%%%%%%%%%%%%%%5
\begin{lemma}[Gluing lemma~\cite{Margot_thesis,KolmanKT:20}] \label{lem:glued_product}
Let $P$ and $Q$ be $0/1$-polytopes and let the $k$ (glued) coordinates
in $P$ be labeled $x_1,\ldots,x_k$, and the $k$
(glued) coordinates in $Q$ be labeled $y_1,\ldots,y_k$. Suppose that
$\mathbf{1}^\intercal \bx \leqslant 1$ is valid for $P$ and $\mathbf{1}^\intercal \by \leqslant 1$ is valid for $Q$. 
Then $\xc{P\times_k Q}\leqslant  \xc{P}+\xc{Q}$.
\end{lemma}

We conclude this section with a lemma about bonds of $G_1\ksum{u}G_2$, 
$G_1\ksum{uv}G_2$ and $G_1\ksum{uv}^{-}G_2$; the claims about bonds of 
$G_1\ksum{u}G_2$ and $G_1\ksum{uv}G_2$ were observed earlier~\cite{Chaourar:20,CJN:23}, the claim about bonds of $G_1\ksum{uv}^{-}G_2$ is new.
For a connected graph $G=(V,E)$ and vertices $u,v \in V$, let $\CC(G)$ denote the set
of all bonds of $G$, $\CC_{uv}(G)$ denote the set of bonds of $G$ with $u$ and $v$ on different
sides of the cut, and $\CC_{\thinbar{uv}}(G)$ the set of bonds of $G$ with $u$ and $v$ on the same side;
if $uv\in E$, we also write just $\CC_e(G)$ and
$\CC_{\bar e}(G)$ instead of $\CC_{uv}(G)$ and $\CC_{\thinbar{uv}}(G)$.
For two sets $\CC_1$ and $\CC_2$ of subsets of edges of a graph $G$,
%let $\CC_1\ksum{2}\CC_2=\{C_1\cup C_2 \ | \ C_1\in \CC_1,\ C_2\in \CC_2\}$
let $\CC_1\utimes\CC_2=\{C_1\cup C_2 \ | \ C_1\in \CC_1,\ C_2\in \CC_2\}$
denote all pairwise unions of the subsets in $\CC_1$ and $\CC_2$.
\begin{lemma}\label{lem:2sum-minus}
Let $G_1=(V_1,E_1), G_2=(V_2,E_2)$ be connected graphs.
\begin{enumerate}
    \item If $V_1\cap V_2=\{u\}$, then $$\CC(G_1 \ksum{u} G_2)=\CC(G_1)\cup \CC(G_2)\ .$$
    \item If $V_1\cap V_2=\{u,v\}$ and $E_1\cap E_2=\{uv\}$, then for $e=uv$
%        $$\CC(G_1 \ksum{uv} G_2)= \CC_{\overline{uv}}(G_1)\cup \CC_{\overline{uv}}(G_2)\cup \left (\CC_{uv}(G_1)\ksum{2} \CC_{uv}(G_2)\right)\ .$$
        \begin{align*} 
        \CC(G_1 \ksum{e} G_2)= \CC_{\bar{e}}(G_1)\cup \CC_{\bar{e}}(G_2)\cup \left (\CC_{e}(G_1)\utimes \CC_{e}(G_2)\right)\ .
        \end{align*}
    \item If $V_1\cap V_2=\{u,v\}$, $E_1\cap E_2=\{uv\}$, $G_1\setminus uv$ is connected and $G_2\setminus uv$ is not, then for $e=uv$
    \begin{align*} 
        \CC(G_1 \ksum{e}^{-} G_2)= \CC(G_1\setminus e)\cup \CC_{\bar{e}}(G_2) \ .
        \end{align*}
    \item If $V_1\cap V_2=\{u,v\}$, $E_1\cap E_2=\{uv\}$, and both $G_1\setminus uv$  and $G_2\setminus uv$ are connected, then for $e=uv$
    \begin{align*} 
        \CC(G_1 \ksum{e}^{-} G_2)= \CC_{\bar{e}}(G_1)\cup \CC_{\bar{e}}(G_2)\cup \left (\CC_{uv}(G_1\setminus e)\utimes \CC_{uv}(G_2\setminus e)\right)\ .
        \end{align*}
\end{enumerate}
\end{lemma}
\begin{proof} 
The four cases %of the Lemma~\ref{lem:2sum-minus} 
are illustrated in Figures~\ref{fig:1sum}-\ref{fig:2sumB}.
For each of them, we show that the left-hand side is a subset of the right-hand side, and
vice versa.
\begin{figure}[bth]
\centering
 \begin{subfigure}{.4\textwidth}
    \includegraphics[scale=0.31]{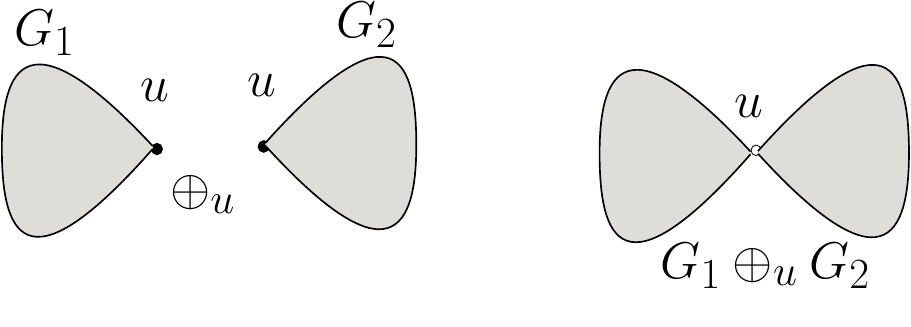}
    \caption{Case 1.}
    \label{fig:1sum}
 \end{subfigure}
 \hfill
 \begin{subfigure}{.4\textwidth}
   \includegraphics[scale=0.31]{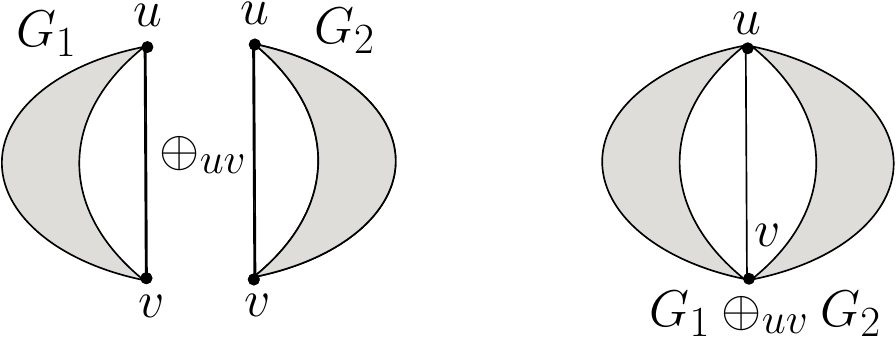}
   \caption{Case 2.}
   \label{fig:2sum}
 \end{subfigure}\\
 \begin{subfigure}{.4\textwidth}
    \includegraphics[scale=0.31]{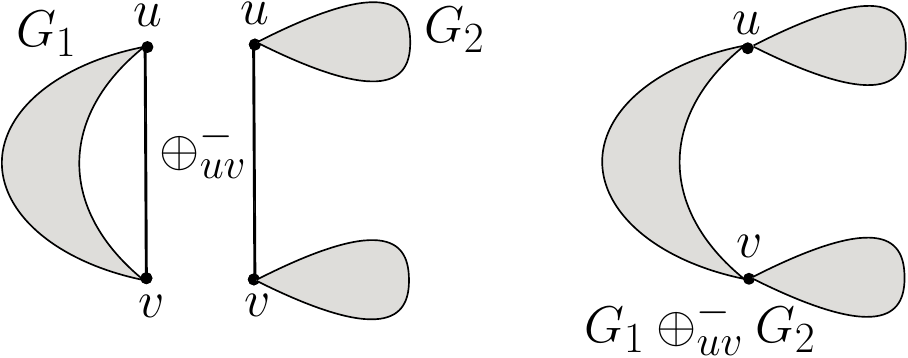}
    \caption{Case 3.}
    \label{fig:2sumA}
 \end{subfigure}
 \hfill
 \begin{subfigure}{.4\textwidth}
    \includegraphics[scale=0.31]{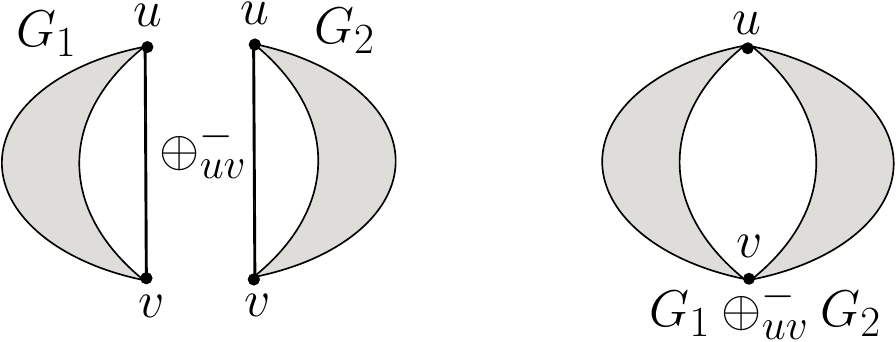}
    \caption{Case 4.}
    \label{fig:2sumB}
 \end{subfigure}
 \hfill
 \caption{$\{1,2\}$-sum of graphs $G_1$ and $G_2$}
\end{figure}

Case 1. Consider $F\in \CC(G_1 \ksum{u} G_2)$. As both sides of the bond $F$ are connected,
either $F\subseteq E_1$ or $F\subseteq E_2$; in the former case, $F\in \CC(G_1)$, in the
later case, $F\in \CC(G_2)$. 
On the other hand, if $F\in \CC(G_1)$ or $F\in \CC(G_2)$,
then the removal of $F$ from $G_1 \ksum{u} G_2$ separates it into two connected parts, that is,
$F$ is a bond of it.

Case 2. Consider first $F\in \CC_{\bar e}(G_1 \ksum{e} G_2)$. As both sides of the bond $F$
are connected and $u$ and $v$ are on the same side of it, we conclude that either 
$F\in \CC_{\bar e}(G_1)$ or $F\in \CC_{\bar e}(G_2)$. If  $F\in \CC_{e}(G_1 \ksum{e} G_2)$,
then $u$ and $v$ are on different sides of the bond $F$. Thus, $F\cap E_1$ must be a bond of 
$G_1$ and $F\cap E_2$ a bond of $G_2$.
On the other hand, if $F\in \CC_{\bar{e}}(G_1)$ or $F\in \CC_{\bar{e}}(G_2)$ or $F\in \CC_{e}(G_1)\utimes \CC_{e}(G_2)$, then $F$ is a bond 
of $G_1 \ksum{e} G_2$.

Case 3. Consider first $F\in \CC_{\thinbar{uv}}(G_1 \ksum{e}^- G_2)$. % with $u$ and $v$ on the same side.  
As $G_2\setminus e$ is disconnected by our assumption, either 
$F\in \CC_{\thinbar{uv}}(G_1\setminus e)$ % with $u$ and $v$ on the same side
%$F$ is a bond of $G_1\setminus e$ with $u$ and $v$ on the same side
or $F\in \CC_{\bar{e}}(G_2)$. 
If $F\in \CC_{uv}(G_1 \ksum{e}^- G_2)$, % with $u$ and $v$ on different sides, 
again exploiting the assumption that $G_2\setminus e$ is disconnected,
we conclude that 
$F\in \CC_{uv}(G_1\setminus e)$.    % with $u$ and $v$ on different sides.
%$F$ is a bond of $G_1\setminus e$ with $u$ and $v$ on different sides.
Combining the two subcases yields that
$\CC(G_1 \ksum{e}^{-} G_2)\subseteq  \CC(G_1\setminus e)\cup \CC_{\bar{e}}(G_2)$.

For the opposite inclusion in Case 3, note that each of the two connected components of $G_2\setminus\nolinebreak e$ shares exactly one vertex 
with $G_1\setminus e$. Thus, every bond of $G_1\setminus e$ is a bond of $G_1 \ksum{e}^- G_2$,
and also every bond $F\in \CC_{\bar{e}}(G_2)$ is a bond of $G_1 \ksum{e}^- G_2$
(cf.~Fig.~\ref{fig:2sumA}), that is, $\CC(G_1\setminus e)\cup \CC_{\bar{e}}(G_2) \subseteq \CC(G_1 \ksum{e}^{-} G_2)$.

Case 4. For the ``$\subseteq$" inclusion we again distinguish two subcases. 

If $F\in \CC_{\thinbar{uv}}(G_1 \ksum{e}^{-} G_2)$,    %with $u$ and $v$ on the same side, 
then one of the subgraphs
$G_1\setminus e$ and $G_2\setminus e$ is untouched by $F$, and $F$
is a bond of the other subgraph with $u$ and $v$ on the same side; note
that $\CC_{\thinbar{uv}}(G_i\setminus e)\subseteq \CC_{\bar{e}}(G_i)$, for $i=1,2$. 

If $F\in \CC_{uv}(G_1 \ksum{e}^{-} G_2)$,  % with $u$ and $v$ on different sides, 
then the assumption of connectivity
of $G_i\setminus e$ implies that $F\cap E_i$ is a bond of 
$G_i\setminus e$ with $u$ and $v$ on different sides, for $i=1,2$. Thus,
$\CC(G_1 \ksum{e}^{-} G_2)\subseteq \CC_{\bar{e}}(G_1)\cup \CC_{\bar{e}}(G_2)\cup \left (\CC_{uv}(G_1\setminus e)\utimes \CC_{uv}(G_2\setminus e)\right)$.

For the opposite inclusion in Case 4, assume first that $F\in \CC_{\bar{e}}(G_1)$.   % with $u$ and $v$ on the same side. 
As $G_2\setminus e$ is connected and shares with $G_1$ the vertices $u$ and $v$ only, 
$(G_1\ksum{e}^{-}G_2)\setminus F$ has the same number of connected components as $G_1\setminus F$,
namely two, so $F$ is a bond of $G_1\ksum{e}^{-}G_2$. The same argument applies if 
$F\in \CC_{\bar{e}}(G_2)$. 

Assume now that $F_i\in\CC_{uv}(G_i\setminus e)$, for $i=1,2$, and $F=F_1\cup F_2$.
Let $H_{iu}$ denote the component of $(G_i\setminus e)\setminus F_i$ containing $u$ and $H_{iv}$ the component containing $v$, for $i=1,2$.
Then $(G_1 \ksum{e}^{-} G_2)\setminus F$ consists of two connected components, $H_{1u}\ksum{u}H_{2u}$ and $H_{1v}\ksum{u}H_{2v}$, meaning that $F$ is a bond of the graph.
\qed
\end{proof}

%%%%%%%%%%%%%%%%%%%%%%%%%%%%%%%%%%%%%%%%%%%%%%%%%%%%%%%%%%%%%%%%%%%%%%%%%%%%%%%%%%%%%%%%%%%%%
\section{The Subdirect Sum of Polytopes}
%%%%%%%%%%%%%%%%%%%%%%%%%%%%%%%%%%%%%%%%%%%%%%%%%%%%%%%%%%%%%%%%%%%%%%%%%%%%%%%%%%%%%%%%%%%%%
In the following we will use matrices and sets of (row) vectors simultaneously with the following interpretation. A matrix $\bB$ will sometimes be viewed as a set of vectors where each row of $\bB$ is a member of the set. Conversely, a set of row vectors $\{\bb_1,\ldots,\bb_m\}$ will be, when convenient, viewed as a matrix $\bB$ with $m$ rows where the $i$-th row is $\bb_i.$

\begin{lemma}\label{lem:direct-sum-facets} 
    Let $P_1=\conv{\bV_1}=\{\bx~|~\bA_1\bx\leqslant \mathbf{0}, \bB\bx\leqslant \mathbf{1}\}$ and $P_2=\conv{\bV_2}=\{\by~|~\bA_2\by\leqslant \mathbf{0}, \bC\by\leqslant \mathbf{1}\}$ be polytopes. Then,
    \begin{align}\label{eq:subdirectsum}
    \conv{\begin{matrix}\bV_1&\mathbf{0}\\\mathbf{0} & \bV_2\end{matrix}}=\left\{\left.\begin{pmatrix}\bx\\\by\end{pmatrix}~\right|~\begin{matrix*}[l]\bA_1\bx\leqslant\mathbf{0},\bA_2\by\leqslant\mathbf{0}&\\
    \bb\bx+\bc\by\leqslant 1&\forall \bb\in\bB, \forall\bc\in\bC\end{matrix*}\right\},
    \end{align}
    where the left-hand side is a shorthand of 
$\conv{\{(\bv,\mathbf{0})| \bv\in \bV_1\}\cup \{(\mathbf{0},\bv)| \bv\in \bV_2\}}$.

\end{lemma}
\begin{proof}
Let $P$ and $P'$, resp., denote the polytopes on the left and right, resp., sides of the equality~(\ref{eq:subdirectsum}).
We start by showing that $P\subseteq P'$. Consider $\binom{\bx}{\by}\in P$.
%We start by showing that $P\subseteq P'$. Consider $(\bx,\by)^\intercal\in P$.
%We start by showing that $P\subseteq P'$. Consider $\begin{pmatrix}\bx\\\by\end{pmatrix}\in P$.
Then $\bx^\intercal=\sum_{\bu\in \bV_1} \lambda_{\bu} \bu$,  $\by^\intercal=\sum_{\bv\in \bV_2} \lambda_{\bv} \bv$ for some 
nonnegative coefficients %$\lambda_{\bu}$, ${\bu}\in \bV_1$ and 
$\lambda_{\bz}$, $\bz\in \bV_1\cup \bV_2$, 
such that $\sum_{\bu\in \bV_1} \lambda_{\bu} + \sum_{\bv\in \bV_2} \lambda_{\bv}=1$.
As for each $\bu\in \bV_1$ and $\bb\in \bB$ we have $\bA_1\bu^\intercal\leqslant \mathbf{0}$ and 
$\bb\bu^\intercal\leqslant 1$, and analogously, 
for each $\bv\in \bV_2$ and $\bc\in \bC$ we have $\bA_2\bv^\intercal\leqslant \mathbf{0}$ and 
$\bc\bv^\intercal\leqslant 1$, we have $\bA_1\bx\leqslant\mathbf{0}$, $\bA_2\by\leqslant\mathbf{0}$
and $\bb\bx+\bc\by\leqslant 1$. Thus, $P\subseteq P'$.

Consider now $\binom{\bx}{\by}\in P'$. If for some $\bb\in\bB$, $\bb\bx \geq 0$,
then for every $\bc\in \bC$, $\bc\by\leq 1$, that is, $\bC\by\leqslant \mathbf{1}$; 
similarly, if for some $\bc\in\bC$, $\bc\by \geq 0$, then for every $\bb\in \bB$, 
$\bb\by\leq 1$, that is, $\bB\bx\leqslant \mathbf{1}$.
Thus, to prove that $P'\subseteq P$, it suffices to show that for some $\bb\in\bB$, $\bb\bx \geq 0$
and for some $\bc\in\bC$, $\bc\by \geq 0$; note that the inequalities $\bA_1\bx\leqslant\mathbf{0}$ 
and $\bA_2\by\leqslant\mathbf{0}$ are always satisfied for our $\bx$ and $\by$.

Assume, for a contradiction, that for every $\bb\in\bB$, $\bb\bx < 0$. Then not only $\bx\in P_1$,
but for every non-negative $\lambda$, also $\lambda \bx \in P_1$. However, this is a contradiction with
the fact that $P_1$ a polytope. Thus, there exists $\bb\in\bB$  such that $\bb\bx \geq 0$.
By the same arguments, there exists $\bc\in\bC$  such that $\bc\by \geq 0$. This completes
the proof of the lemma. \qed
\end{proof}

The above operation is called a \emph{subdirect sum} of $P_1$ and $P_2$
%and the inequalities follow from Proposition 2.3 in \cite{McMullen:76}.
and if the inequalities for $P_1$ and $P_2$ in Lemma \ref{lem:direct-sum-facets} are facet-defining then so are the inequalities of the subdirect sum.

%%%%%%%%%%%%%%%%%%%%%%%%%%%%%%%%%%%%%%%%%%%%%%%%%%%%%%%%%%%%%%%%%%%%%%%%%%%%%%%%%%
%\section{Extension Complexity of the Bond Polytope}
\section{The Bond Polytope}
%%%%%%%%%%%%%%%%%%%%%%%%%%%%%%%%%%%%%%%%%%%%%%%%%%%%%%%%%%%%%%%%%%%%%%%%%%%%%%%%%%%%%%%%%%%%%
Let $G=(V,E)$ be a graph, and let $E'\subseteq\binom{V}{2}\setminus E$ be a subset of non-edges of $G$. An \emph{Augmented Bond Polytope} $\abondp{G}{E'}$ is the convex hull of vectors
$\bx\in\BR^{|E|+|E'|}$ where $\bx_E$ is the characteristic vector of a bond $F\subseteq E$ in $G$ and for every $uv\in E'$,
$\bx_{uv}=0$ if $u$ and $v$ are on the same side of the bond $F$ and $\bx_{uv}=1$ if $u$ and $v$ are on different sides of the bond $F$. We consider the augmented bond polytope due to the fact that when looking at the $2$-sum operation $\ksum{uv}^{-}$,
the bond polytope without $uv$ as an augmented coordinate does not behave well. In particular, an extended formulation for the bond polytope of the resulting graph can only be recursively constructed if the behaviour of vertices $u,v$ in each of the bond of constituent graphs is stored.

%%%%%%%%%%%%%%%%%%%%%%%%%%%%%%%%%%%%%%%%%%%%%%%%%%%%%%%%%%%%%%%%%%%%%%%%%%%%%%%%%%%%%%%%%%%%%
\subsection{Description of the Bond Polytope under $k$-sums $\ksum{1}$ and $\ksum{2}$ }
%\subsection{Description of the Bond Polytope under $\{1,2\}$-sum}
\begin{lemma}\label{lem:1sum-vertices}
Let $G_1=(V_1,E_1)$ and $G_2=(V_2,E_2)$ be two graphs with $V_1\cap V_2=\{v\}$,
and let $E_1'\subseteq \binom{V_1}{2}\setminus E_1, E_2'\subseteq \binom{V_2}{2}\setminus E_2$. Let $\abondp{G_1}{E_1'}$, and $\abondp{G_2}{E_2'}$ be their respective augmented bond polytopes. Suppose $\abondp{G_1}{E_1'}=\conv{\bB_1}$ and  $\abondp{G_2}{E_2'}=\conv{\bB_2}$. Then, 
%\vspace{-12pt}    
\[\abondp{G_1\ksum{v}G_2}{E_1'\cup E_2'}=\conv{\begin{matrix}
    \bB_1 & \mathbf{0} \\ \mathbf{0} & \bB_2
\end{matrix}} \ . \] 
\end{lemma}
\begin{proof}
    Notice  (cf. Lemma~\ref{lem:2sum-minus}, Case 1) that any cut in $G_1\ksum{v}G_2$ either cuts $G_1$ in two components placing all of $G_2$ in the component containing the common vertex $v$ or it cuts $G_2$ in two components placing all of $G_1$ in the component containing the common vertex $v$. Therefore, any bond in $G_1\ksum{v}G_2$ is either a bond in $G_1$ extended with zeroes at the coordinates $x_{uw}$ for $uw\in E_2\cup E'_2$, or a bond in $G_2$ extended with zeroes at the coordinates $x_{uw}$ for $uw\in E_1\cup E'_1$. \qed
\end{proof}

Note that the above description of \abondp{G_1\ksum{v}G_2}{E_1'\cup E_2'} is the subdirect sum of \abondp{G_1}{E_1'} and \abondp{G_2}{E_2'} and thus it can have a number of inequalities that is asymptotically the product of the number of inequalities describing the two multiplicands (cf. Lemma \ref{lem:direct-sum-facets}). Taking $E_1'=E_2'=\emptyset$ we get the same statement for $\bondp{G_1\ksum{1}G_2}.$

More concretely, take any graph $G$ with $k$ vertices such that $\bondp{G}$ has at least $c$ ($d$, resp.) facet-defining inequalities with right hand side $0$ ($1$, resp.). Set $G_0=G$ and let $G_i$ be obtained by taking a $1$-sum of $G_{i-1}$ and $G$ (over an arbitrarily chosen vertex). Then $G_n$ has $(k-1)n+1$ vertices but $\bondp{G_n}$ has at least $cn+d^n$ facet-defining inequalities due to Lemma \ref{lem:direct-sum-facets} and the subsequent comments. Since \bondp{K_4} has $4$ facet-defining inequalities with right hand side $1$ and $12$ facet-defining inequalities with right hand side $0$ \cite{CJN:23}, %theorem 8.2
taking $G$ to be $K_4$ we even get $(K_5\setminus e)$-minor-free graphs with $3n+1$ vertices whose bonds have $12n+4^n$ facets, for arbitrary $n$.

Therefore, in general, one cannot obtain a linear size description for $(K_5\setminus e)$-minor-free graphs unless one is willing to consider extended formulations.

\begin{lemma}\label{lem:2sum-vertices}
Let $G_1=(V_1,E_1)$ and $G_2=(V_2,E_2)$ be two graphs with $E_1\cap E_2=\{e\}$, and let $E_1'\subseteq\binom{V_1}{2}\setminus E_1, E_2'\subseteq\binom{V_2}{2}\setminus E_2$. Let $\abondp{G_1}{E_1'}$, and $\abondp{G_2}{E_2'}$ be their respective augmented bond polytopes. Suppose $\abondp{G_1}{E_1'}=\conv{\bB_1}$ and  $\abondp{G_2}{E_2'}=\conv{\bB_2}$. Then, 
\[\abondp{G_1\ksum{e}G_2}{E_1'\cup E_2'}=\mathrm{conv}\left\{{\begin{pmatrix}
    \bB_1^0 & \mathbf{0} \\ \mathbf{0} & \bB_2^0
\end{pmatrix}} \cup (\bB_1^1 \times \bB_2^1)\ \right \},\]
where $\bB_j^i=\{\bb\in \bB_j ~|~ \bb_e=i\}$ for $i\in\{0,1\}$ and $j\in\{1,2\}$
and $\times$ denotes the Cartesian product. 
%$$\conv{ \begin{matrix*}[l] \{(\bb_1,\mathbf{0}_{~})\in\bB_1\times\{\mathbf{0}\}~|~ (\bb_1)_e=0\} & \cup\\  \{(\mathbf{0}_{~},\bb_2)\in\{\mathbf{0}\}\times\bB_2~|~ (\bb_2)_e=0\}& \cup\\ \{(\bb_1,\bb_2)\in\bB_1\times\bB_2~~|~(\bb_1)_e=(\bb_2)_e=1\}&  \end{matrix*}}.$$
%where $e$ is the edge along which the $2$-sum is taken.
\end{lemma}
\begin{proof}
Let $e=\{u,v\}$. Any bond in $G=G_1\ksum{e}G_2$ either has $u,v$ in the same component or in different components. If $u,v$ are in the same component, then the bond is obtained either from a bond of $G_1$ by putting $G_2$ entirely in the component containing $u,v$ or from a bond of $G_2$ by putting $G_1$ entirely in the component containing $u,v$. %(cf. Fig.~\ref{fig:sum}). 
If $u,v$ are in different components, then the bond is obtained from a bond $\bb_1$ of $G_1$ and a bond $\bb_2$ of $G_2$ such that $u,v$ are in different components of both $\bb_1$ and $\bb_2$ (cf. Lemma~\ref{lem:2sum-minus}, Case 2). \qed
\end{proof}

%%%%%%%%%%%%%%%%%%%%%%%%%%%%%%%%%%%%%%%%%%%%%%%%%%%%%%%%%%%%%%%%%%%%%%%%%%%%%%%%%%%%%%%%%%%%%
\subsection{Extension Complexity of Bond Polytope under $k$-sums
$\ksum{1}$, $\ksum{2}$ and $\ksum{2}^-$}
%\subsection{Extension Complexity of the Bond Polytope under $\{1,2\}$-sum}

\begin{theorem}\label{thm:2sum_xc}
Let $G_1=(V_1,E_1)$ and $G_2=(V_2,E_2)$ be two graphs and let $E_1'\subseteq \binom{V_1}{2}\setminus E_1, E_2'\subseteq \binom{V_2}{2}\setminus E_2$. Let $\abondp{G_1}{E_1'}$, $\abondp{G_2}{E_2'}$ be their respective augmented bond polytopes. 
%Suppose $\abondp{G_1}{E_1'}=\conv{\bB_1}$ and  $\abondp{G_2}{E_2'}=\conv{\bB_2}$. 
Then, 
{\small\[\xc{\abondp{G_1\ksum{k}G_2}{E_1'\cup E_2'}}\leqslant \xc{\abondp{G_1}{E_1'}}+\xc{\abondp{G_2}{E_2'}}+\Oh{1},\]}
for $k\in\{1,2\}$. Furthermore, given extended formulations for \abondp{G_1}{E_1'} and \abondp{G_2}{E_2'}, an extended formulation for \abondp{G_1\ksum{k}G_2}{E_1'\cup E_2'} can be constructed in linear time.
\end{theorem}
\begin{proof}
    For $k=1$, the result is an immediate corollary of Lemma \ref{lem:1sum-vertices} and Theorem~\ref{thm:balas}. For $k=2$, let $e$ be the edge along which the $2$-sum is taken, and let $d$ be the dimension %of the affine hull
    of $\abondp{G_1}{E_1'}$. 
    
    First, we assume that \abondp{G_1}{E_1'} is embedded in $\BR^{d+2}$. Call the extra two coordinates $w,z$ and the embedded polytope $P_{G_1}$. We assume that the following property holds for each $\bb\in\vertices{P_{G_1}}$:
    \[\begin{matrix*}
        (\bb)_e=0 & \implies & (\bb)_w=0, (\bb)_z=1\\ 
        (\bb)_e=1 & \implies & (\bb)_w=0, (\bb)_z=0 
    \end{matrix*}\]  
   This can be achieved by taking the glued product of $\abondp{G_1}{E_1'}$  with the line segment $S=\conv{\{(0,0,1),(1,0,0)\}}$, glueing coordinate $(\bx)_e$ in \abondp{G_1}{E_1'} with the first coordinate of the segment $S$. This results in an additive $\Oh{1}$ increase in the extension complexity due to 
   Lemma~\ref{lem:glued_product}. Next, we take the convex hull of the union of the resulting polytope and the point $(\mathbf{0},1,0)\in \BR^{\mathrm{d}+2}$ to obtain $P'_{G_1}$, again resulting in an $\Oh{1}$ additive increase in the extension complexity due to Theorem~\ref{thm:balas}. 
   Altogether, we have $\xc{P'_{G_1}}\leqslant \xc{\abondp{G_1}{E_1'}}+\Oh{1}.$

     Similarly, we obtain $P'_{G_2}$ from \abondp{G_2}{E_2'} by adding new coordinates $w,z$ first ensuring 
    \[\begin{matrix*}
        (\bb)_e=0 & \implies & (\bb)_w=1, (\bb)_z=0\\
        (\bb)_e=1 & \implies & (\bb)_w=0, (\bb)_z=0
    \end{matrix*}\]
    for each $\bb\in\vertices{P_{G_2}},$    
    and then taking the convex hull of the union of the resulting polytope with the point $(\mathbf{0},0,1).$ By the same arguments as for $P'_{G_1}$ we have that $\xc{P'_{G_2}}\leqslant \xc{\abondp{G_2}{E_2'}}+\Oh{1}.$

    Finally, we take the glued-product of $P'_{G_1}$ and $P'_{G_2}$ where the gluing is done over the $z,w$ coordinates in $P'_{G_1}$ with the $z,w$ coordinates in $P'_{G_2}$. The resulting polytope is an extended formulation of \abondp{G_1\ksum{2}G_2}{E_1'\cup E_2'} by Lemma~\ref{lem:2sum-vertices}, and by Lemma~\ref{lem:glued_product}, it has extension complexity at most $\xc{P'_{G_1}}+\xc{P'_{G_2}}+\Oh{1}$ which is $\xc{\abondp{G_1}{E_1'}}+\xc{\abondp{G_2}{E_2'}}+\Oh{1}.$ Note that all the steps in the proof are efficiently constructive so the resulting extended formulation can be constructed in linear time.\qed
\end{proof}

\begin{lemma}\label{lem:2sum_removed_xc_cases_1_2}
    Let $G_1=(V_1,E_1)$ and $G_2=(V_2,E_2)$ be two graphs such that $\{u,v\}= V_1\cap V_2$ and $\{uv\}= E_1\cap E_2$. Let $E_1'\subseteq \binom{V_1}{2}\setminus E_1$ and $E_2'\subseteq\binom{V_2}{2}\setminus E_2$. Suppose that $u$ and $v$ belong to the same connected component in $G_1\setminus uv$. 
    \begin{enumerate}
      \item If $G_2\setminus uv$ is disconnected, then
\vspace{-6pt}    
    \begin{multline*}
    \xc{\abondp{G_1\ksum{uv}^{-}G_2}{E_1'\cup E_2'\cup\{uv\}}} \leqslant \\\xc{\abondp{G_1\setminus uv}{E_1'\cup \{uv\}}}+\xc{\abondp{G_2}{E_2'}}+\Oh{1}.   \end{multline*}
%\vspace{-18pt}    
    \item If $G_2\setminus uv$ is connected, then
\vspace{-6pt}    
    \begin{multline*}
        \xc{\abondp{G_1\ksum{uv}^{-}G_2}{E_1'\cup E_2'\cup\{uv\}}} \leqslant \\\xc{\abondp{G_1\setminus uv}{E_1'\cup \{uv\}}}+\xc{\abondp{G_2\setminus uv}{E_2'\cup\{uv\}}}+\Oh{1}.
    \end{multline*}
    \end{enumerate}
    Furthermore, in each of these cases the resulting extended formulation can be constructed in linear time given extended formulations for the appropriate polytopes.
\end{lemma}
\begin{proof}
    The claim of Part 1 follows immediately from the characterization in Lemma \ref{lem:2sum-minus} (Case 3) and Theorem~\ref{thm:balas}. For Part 2, we note the characterization in Lemma \ref{lem:2sum-minus} (Case 4) and observe that the case is analogous to that in Theorem~\ref{thm:2sum_xc} and hence an analogous proof yields the result.
    \qed
\end{proof}

%%%%%%%%%%%%%%%%%%%%%%%%%%%%%%%%%%%%%%%%%%%%%%%%%%%%%%%%%%%%%%%%%%%%%%%%%%%%%%%%%%%%%%%%%%%%%
\subsection{Extension~Complexity~of~the~Bond~Polytope~for~$(K_5\setminus e)$-minor-free~Graphs}

\begin{lemma}
\label{lem:abond_xc_wheel}
    Let $G=(V,E)$ be a graph and let $E'\subseteq \binom{V}{2}\setminus E$ such that $G'=(V,E\cup E')$ is a wheel graph. Then, $\xc{\abondp{G}{E'}}=\Oh{|V|}.$ Furthermore, such an extended formulation can be constructed in time $\Oh{|V|}.$
\end{lemma}
\begin{proof}
We prove the lemma by induction on the number of vertices on the rim. For any constant $n$ we have a constant size graph and so it has a constant number of bonds and the extension complexity of any augmented bond polytope is a constant.

Suppose the claim holds for $n$. That is, for any subgraph of a wheel $W_n$ with center $c$ and rim vertices $0,\ldots, n$ the extension complexity of the corresponding augmented bond polytope (with the non-edges as augmented coordinates) is $\Oh{n}.$ Now consider a subgraph $G_{n+1}$ of $W_{n+1}$ with a set of non-edges $\overline{E}_{n+1}.$ 

Either there exists a rim-vertex with a degree strictly less than three, or the given graph is the wheel itself and there are no augmented coordinates and thus the extension complexity of $\abondp{G}{\overline{E}_{n+1}}$ is linear and can be explicitly described \cite{CJN:23}. Similarly, either there exists a rim-vertex with a degree strictly more than one, or the given graph is a star and the augmented coordinates correspond to the cycle $0,\ldots, n+1$. The star has only a linear
number of bonds - one for each ray being cut off. So an extended formulation can be constructed in linear time by Theorem \ref{thm:balas}, using each bond explicitly.

Therefore, without loss of generality, we can assume that there is a rim vertex that has a degree exactly two. For simplicity, we assume that this vertex is labelled $n+1$ even though for the construction of an extended formulation the actual label of such a vertex can be directly used. We distinguish the following cases: 
\begin{enumerate}
    \item[1.] edges $\{n,n+1\}, \{c,n+1\}$ are present while the edge $\{0,n+1\}$ is absent, 
    \item[2.] edges $\{0,n+1\}, \{c,n+1\}$ are present while the edge $\{n,n+1\}$ is absent, %and
    \item[3.] edges $\{0,n+1\}, \{n,n+1\}$ are present while the edge $\{c,n+1\}$ is absent. 
\end{enumerate}
    Cases 1 and 2 are identical except for vertex labeling so we will consider
    only case 1. We construct a subgraph $G_n$ of $W_n$ as follows. We remove
    vertex $n+1$. For each $i\in\{0,n\}$, we keep $\{c,i\}$ as an edge 
    in $G_n$ if it was an edge in $G_{n+1}$, otherwise we keep it as a non-edge
    in $\overline{E}_n$. Similarly for $\{i,i+1\}$, for $i\in\{0,n-1\}$.
    Finally, $\{0,n\} \in \overline{E}_n.$ By our inductive hypothesis the
    augmented bond polytope $\abondp{G_n}{\overline{E}_n}$ has extension
    complexity $\Oh{n}.$ 
    
    Note that any bond in $G_{n+1}$ must have vertex $n+1$ in the same component
    as either $c$ or $n$. The only exception is the bond that has vertex $n+1$
    as one component and the rest of the graph in the other component. Since,
    every unexceptional bond in $G_{n+1}$ corresponds to a bond in $G_n$
    and there are only finitely many types of bonds -- in how they appear at the
    vertices $c, n, n+1, $ and $0$, we glue extra coordinates onto
    $\abondp{G_n}{\overline{E}_n}$ encoding this. More specifically, consider
    the glued product of $\abondp{G_n}{\overline{E}_n}$ over the coordinates
    $x_{c,n},x_{c,0},x_{0,n}$ with Polytope \ref{polytope1}.
     \begin{equation*}
          %\conv
          {\begin{pmatrix} x_{c,n} & x_{0,n} & x_{c,0} & x_{n,n+1} &
        x_{n+1,0} & x_{c,n+1} \\ 0 & 0 & 0 & 0 & 0 & 0 \\ 0 & 1 & 1 & 0 & 1 & 0
        \\ 1 & 0 & 1 & 1 & 1 & 0 \\ 1 & 0 & 1 & 0 & 0 & 1 \\ 1 & 1 & 0 & 1 & 0 &
        0 \\ 1 & 1 & 0 & 0 & 1 & 1 \end{pmatrix}}
    \end{equation*}
    \vspace{-.4cm}
    \captionof{mat}{Vertices}\label{polytope1}  
    Finally, we add the single exceptional bond and the resulting polytope is an extended formulation for $\abondp{G_{n+1}}{\overline{E}_{n+1}}$.

    Similarly for case 3, we obtain $G_{n}$ by keeping edges/non-edges as they are in $G_{n+1}$ over vertices $c,0,\ldots,n$ and $\{0,n\}\in E(G_n).$ Similar to the previous case we glue extra coordinates to extend bonds in $G_n$ to bonds in $G_{n+1}$. We see that the polytope needed for glueing is Polytope \ref{polytope2}.  
%%% Petr June 2 3.2 -> 4.2
%\vspace{-25pt}
%      \begin{multicols}{2}
%      \begin{equation*}
%          %\conv
%          {\begin{pmatrix} x_{c,n} & x_{0,n} & x_{c,0} & x_{n,n+1} &
%        x_{n+1,0} & x_{c,n+1} \\ 0 & 0 & 0 & 0 & 0 & 0 \\ 0 & 1 & 1 & 0 & 1 & 0
%        \\ 1 & 0 & 1 & 1 & 1 & 0 \\ 1 & 0 & 1 & 0 & 0 & 1 \\ 1 & 1 & 0 & 1 & 0 &
%        0 \\ 1 & 1 & 0 & 0 & 1 & 1 \end{pmatrix}}
%      \end{equation*}\captionof{mat}{Vertices}\label{polytope1}  
      \begin{equation*}
          %\conv
          {\begin{pmatrix} x_{c,n} & x_{0,n} & x_{c,0} & x_{n,n+1} &
        x_{n+1,0} & x_{c,n+1} \\ 0 & 0 & 0 & 0 & 0 & 0 \\ 0 & 1 & 1 & 0 & 1 & 0
        \\ 1 & 0 & 1 & 1 & 1 & 0 \\ 1 & 0 & 1 & 0 & 0 & 1 \\ 1 & 1 & 0 & 1 & 0 &
        0 \\ 1 & 1 & 0 & 0 & 1 & 1 \end{pmatrix}}
      \end{equation*}
      \vspace{-.4cm}
      \captionof{mat}{Vertices}\label{polytope2}
%      \end{multicols}

   % {{ $$\conv{\begin{matrix}
   %      x_{c,n} & x_{0,n} & x_{c,0} & x_{n,n+1} & x_{n+1,0} & x_{c,n+1} \\
   %      0 & 0 & 0 & 0 & 0 & 0 \\
   %      0 & 1 & 1 & 0 & 1 & 0 \\
   %      0 & 1 & 1 & 1 & 0 & 1 \\
   %      1 & 0 & 1 & 1 & 1 & 0 \\
   %      1 & 0 & 1 & 0 & 0 & 1 \\
   %      1 & 1 & 0 & 1 & 0 & 0 \\
   %      1 & 1 & 0 & 0 & 1 & 1
   %  \end{matrix}}.$$}}
    Finally, we add the single exceptional bond and the resulting polytope is an extended formulation for $\abondp{G_{n+1}}{\overline{E}_{n+1}}$.

    Thus, we see that $\xc{\abondp{G_{n+1}}{\overline{E}_{n+1}}}\leqslant\xc{\abondp{G_n}{\overline{E}_n}}+\Oh{1}.$ Furthermore, observe that a degree two vertex in a subgraph of a wheel can be found in linear time so the inductive step can be performed with a constant overhead.
    This concludes the proof of the inductive step.
    \qed
\end{proof}

\begin{theorem}\label{thm:abond_xc}
    Let $G=(V,E)$ be a graph and let $E'\subseteq \binom{V}{2}\setminus E$ be such that the graph $G'=(V,E\cup E')$ does not contain $(K_5\setminus e)$ as a minor. Then, $\xc{\abondp{G}{E'}}=\Oh{|V|}.$ Furthermore, such an extended formulation can be constructed in time $\Oh{|V|}.$
\end{theorem}

\begin{proof}
If $G$ is not connected then depending on whether it has two or more connected components, there is either only the trivial bond or no bonds. Therefore we may assume that $G$ is connected. Now, by Theorem \ref{thm:decomp}, $G=G_1 \oplus^1 \dots \oplus^{\ell-1} G_\ell$ where each
$G_i$ is isomorphic to a wheel graph, $\prism$, $K_2$, $K_3$, or $K_{3,3}$,
and each operation $\oplus^i$ is $\oplus_1$, $\oplus_2$ or $\oplus_2^-$. We prove the claim by induction on $\ell$.

If $\ell=1$ then $G$ is either $K_2$, $K_3$, $K_{3,3}$, $\prism$ or a wheel graph.
 If $G$ is either $K_2$, $K_3$, $K_{3,3}$ or $\prism$, then it has constant size and hence a constant number of bonds. Thus, the augmented bond polytope $\abondp{G}{E'}$ has constant size. If $G$ is a wheel, then applying Lemma~\ref{lem:abond_xc_wheel} gives us the desired result.  

For the inductive step, let $G'=G_1 \oplus^1 \dots \oplus^{\ell-2} G_{\ell-1}$. Then $G=G'\oplus^{\ell-1} G_\ell$. %Let $E_1'=E'\cap\binom{V(G')}{2}$ and $E_2'=E'\cap\binom{V(G_\ell)}{2}$.
 If $G=G'\ksum{1}G_\ell$, then Lemma \ref{lem:1sum-vertices} %together 
 with Theorem \ref{thm:balas} gives us the desired result. 
 If $G=G'\ksum{e}G_\ell$, then Theorem \ref{thm:2sum_xc} gives us the desired result. 
 Finally, if $G=G'\ksum{e}^{-}G_\ell$, then Lemma \ref{lem:2sum_removed_xc_cases_1_2}, Part 1 or Part 2 --
 %or Lemma \ref{lem:2sum_removed_xc_cases_1_2} -- 
  depending on how the end vertices of $e$ are connected in $G_1\setminus e$ and $G_2\setminus e$ -- gives us the desired result; note that $\xc{\abondp{G}{E\cup\{e\}}}\geqslant\xc{\abondp{G}{E}}.$
  \qed
\end{proof}

Applying Theorem \ref{thm:abond_xc} to $(K_5\setminus e)$-minor-free graphs together with $E'=\emptyset$ we get the following result.

\begin{theorem}
    Let $G=(V,E)$ be a $(K_5\setminus e)$-minor-free graph. Then,  
    $$\xc{\bondp{G}}= \Oh{|V|} \ .$$ 
    Moreover, this extended formulation can be constructed in time $\Oh{|V|}.$
\end{theorem}

%%%%%%%%%%%%%%%%%%%%%%%%%%%%%%%%%%%%%%%%%%%%%%%%%%%%%%%%%%%%%%%%%%%%%%%%%%%%%
\section{Linear Time Algorithm for~$(K_5\setminus e)$-minor-free Graphs}
%%%%%%%%%%%%%%%%%%%%%%%%%%%%%%%%%%%%%%%%%%%%%%%%%%%%%%%%%%%%%%%%%%%%%%%%%%%%%%%%%%%%%%%%%%%%%
\floatname{algorithm}{Procedure}
To solve the \MaxBond problem on $(K_5\setminus e)$-minor-free graphs,
we use the same framework %recursive approach to construct the maximum bond 
as Chimani et al.~\cite{CJN:23} did, based on the %Wagner's theorem and on the 
fact that %the corresponding 
decomposition of a graph into $3$-connected components can be constructed in linear
time due to the algorithm of Hopcroft and Tarjan~\cite{HopcroftT:73}.
The important key difference is that
for the wheel graph, we describe a relatively simple linear time algorithm whereas
Chimani et al. rely on the algorithm for the construction of maximum bonds on
graphs of bounded treewidth~\cite{DLPSS:19}. 
The main idea of our algorithm is simple - 
to mimic Kadane's dynamic programming approach~\cite{Bentley:84} for the Maximum sum 
subarray problem in a slightly more complicated setting. 
\begin{theorem}\label{thm:wheel}
The weighted \MaxBond problem can be solved in time $O(n)$ for the wheel graph $W_n$.
\end{theorem}
\begin{proof}
Given a weighted wheel graph $W_n$, for notational simplicity, we also use
the following notation: for
$i=0,1,\ldots,n-2$, $a_i=w(i,c)$, $b_i=w(i,i+1)$, and
$a_{n-1}=w(n-1,c)$ and $b_{n-1}=w(n-1,0)$.
%\InsertBoxR{0}{\includegraphics[scale=0.3]{figs/kolo.pdf}}{}
\begin{figure}[bth]
\centering
\includegraphics[scale=0.35]{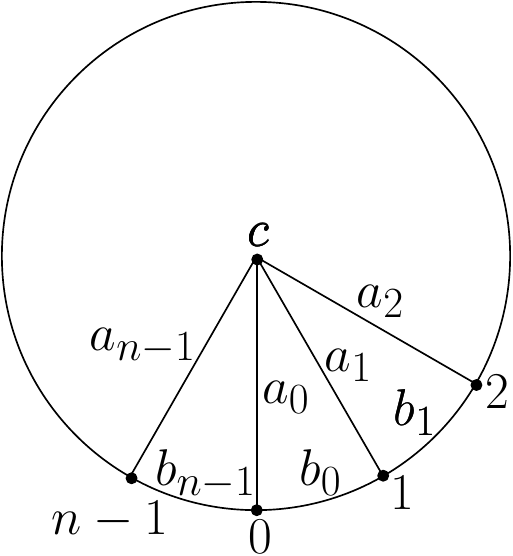}
\caption{The wheel graph $W_n$}
\end{figure}

Given a bond $(S,V\setminus S)$ of $W_n$, its two connected components 
have a very
special form: either one of them is the hub vertex $c$ and the other consists of
all the vertices on the rim - such a bond is called the {\em trivial bond}, 
or one of the connected components is a path on at most $n-1$ vertices of the rim,
denoted $P_S$ in the following,
and the other component consists of all the other vertices which is a fan graph. 

Let $\CF$ be the set of all non-trivial bonds in $W_n$.
Consider the partition of $\CF$ into $\CF_1\cup \CF_2\cup \CF_3$ defined 
below; we refer to bonds in $\CF_i$ as \emph{bonds of type $i$}:
\begin{align*}
\CF_1 &=\{S\in \CF \ | \ \mbox{the path $P_S$ does not contain the edges $\{n-1,0\}$ and $\{0,1\}$}\}\\
\CF_2 &=\{S\in \CF \ | \ \mbox{the path $P_S$ does not contain the edges $\{0,1\}$ and $\{1,2\}$}\}\\
\CF_3 &=\{S\in \CF \ | \ \mbox{the path $P_S$ contains the edge $\{0,1\}$}\}
\end{align*}
%Note that the union $\CF=S_1\cup S_2\cup S_3$.
If we can find for each $i\in\{1,2,3\}$, $\min_{S\in \CF_i}\sum_{e\in \delta(S)}w(e)$, in linear time,
we can solve the bond problem on the wheel graph in linear time.
Also note that the sets $\CF_1$ and $\CF_2$ are of the same kind, just {\em rotated};
thus, it suffices to describe an algorithm for finding the optimal bond from $\CF_1$,
and from $\CF_3$.

%%%%%%%%%%%%%%%%%%%%%%%%%%%%%%%%%%%%%%%%%%%%%%%%
\subsection{Finding the optimal bond from $\CF_1$}
For each $k\in \{1,n-1\}$, we define the following quantities; for most of the quantities we
introduce two names - a full name, indicating its meaning, and an abbreviation:
\begingroup
\addtolength{\jot}{-1em}
\begin{align*}
&\absf(k) = \osf(k)=\max\{b_{i-1}+b_{j}+\sum_{l=i}^{j}a_l\ : \ i\in[1,k],j\in[i,k]\} \\
&\absfl(k) = \osfl(k)= \\
 & \hspace{1.2cm} \min\bigl\{i\in[1,k]  \ :  \ \absf(k)=\max\{b_{i-1}+b_{j}+\sum_{l=i}^{j}a_l\ : \ j\in[i,k]\}\bigr\} \\
&\absfr(k) = \osfr(k)= \\
 & \hspace{1.2cm} \min\{j\in[\absfl(k),k] \ : \ \absf(k)=b_{\absfl(k)-1}+b_{j}+\sum_{l=\absfl(k)}^{j}a_l \} \\
&\asb(k) = \po(k)=\max\{b_{i-1}+b_{k}+\sum_{l=i}^{k}a_l\ : \ i\in[1,k]\} \\
&\asbl(k)=\pol(k)=\min\{i\in[1,k]\ : \ \asb(k)=b_{i-1}+b_{k}+\sum_{l=i}^{k}a_l\}
\end{align*}
\endgroup
In words, $\osf(k)$ is the cost of the maximum bond that cuts out a subpath
of $\{1,2,\ldots,k\}$, and %the numbers 
$\osfl(k)$ and $\osfr(k)$ are the indices $i$ and $j$
for which the maximum value $\osf(k)$ is attained; as there might be more 
such indices, we pick the smallest ones. 
Similarly, $\po(k)$ is the cost of the maximum bond that cuts out a subpath
of $\{1,2,\ldots,k\}$ ending in $k$, and the number $\pol(k)$ is the index
of the vertex in which the subpath of the cost $\po(k)$ starts.
Note that $\osf(n-1)$ is the cost of the maximal bond from the set $\CF_1$ and that it
consists of edges $\{\{\osfl(n-1)-1,\osfl(n-1)\},\{\osfr(n-1),$ $\osfr(n-1)+1\} \}\cup 
\bigcup_{l=\absfl(n-1)}^{\absfr(n-1)}\{l,c\}$.

The quantities can be computed in linear time using dynamic programming 
by Procedure \textsc{Best-Bond-1}.
\begin{algorithm}[tbh]
\caption{\textsc{Best-Bond-1}}
{{\begin{algorithmic}[1]
    \State $\osf(1) \gets b_{0 }+b_1+a_1, \po(1)\gets\osf(1)$ 
    \State $\osfl(1) \gets 1,\ \osfr(1) \gets 1, \ \pol(1) \gets 1$
    \For{$j=1,\ldots,n-2$}
    	\State $\oc \gets \po(j)+a_{j+1}-b_j+b_{j+1}$
	\If{$ \osf(j)\geq \oc $}
		\State $\osf(j+1) \gets \osf(j)$
		\State $\osfl(j+1) \gets \osfl(j)$
		\State $\osfr(j+1) \gets \osfr(j)$
	\Else
		\State $\osf(j+1) \gets \oc$
		\State $\osfl(j+1) \gets \pol(j)$
		\State $\osfr(j+1) \gets j+1$
    	\EndIf
   	\If{$ \po(j) < 0 $}
		\State $\po(j+1) \gets a_{j+1}-b_j+b_{j+1}$
		\State $\pol(j+1) \gets j+1$
	\Else
		\State $\po(j+1) \gets \po(j)+a_{j+1}-b_j+b_{j+1}$
		\State $\pol(j+1) \gets \pol(j)$
    	\EndIf
    \EndFor
    \State \textbf{return}$\left(\osf(n-1),\osfl(n-1),\osfr(n-1)\right)$
\end{algorithmic}}}
\end{algorithm}
The correctness of the procedure is ensured by the following Claim.
\begin{claim}
For every $k=1,\ldots,n-1$, the values computed by the Procedure \textsc{Best-Bond-1}
are the correct values for $\absf(k), \absfl(k), \absfr(k), \asb(k), \asbl(k)$.
\end{claim}
\begin{proof}
The proof is by induction on $k$. For $k=1$, there is only one possibility for the bond
that corresponds to a subpath of $\{1\}$, namely the bond $S=\{1\}$. Thus, the correct values
of the quantities are $\absf(1)=\asb(1)=b_0 + b_1 + a_1, \absfl(1)=\absfr(1)=\asbl(1)=1$
which is exactly what the procedure computes in steps 1 and 2.

Inductive step. We assume that the procedure correctly computed all the values up to
the index $k$ and we want to prove that the values with index $k+1$ are computed correctly 
as well. We start with the values $\absf, \absfl, \absfr$.
Let $(S,V\setminus S)$ be the lexicographically smallest maximum bond, where $S=\{i,\ldots,j\}\subseteq \{1,\ldots,k+1\}$. Then we have that either $j\leq k$, or $j=k+1$. In the first case,
$\absf(k+1)=\absf(k)$, $\absfl(k+1)=\absfl(k)$, and $\absfr(k+1)=\absfr(k)$, and the 
procedure computes these values in steps 6-8. In the second case, $\absf(k+1)=\asb(k)+a_{k+1}-b_k+b_{k+1}$, $\absfl(k+1)=\po(k)$, $\absfr(k+1)=k+1$, and the procedure computes these values
in steps 4 and 10-12.

Consider now the values $\asb$ and $\asbl$. We distinguish two cases:
$\asb(k)<0$, and $\asb(k)\geq 0$. In the first case, the maximum bond of the desired form
consists of the vertex $k+1$ only, and the correct values are $\asb(k+1)=a_{k+1}-b_k+b_{k+1}$
and $\asbl(k+1)=k+1$ which is what the procedure computes in steps 13-15. In the other case,
the maximum bond is of the form $S=\{\asbl(k),\ldots,k+1\}$ and of cost $\asbl(k)+a_{k+1}-b_{k}+k_{k+1}$; the procedure computes these values in steps~17-18.
\qed
\end{proof}

%%%%%%%%%%%%%%%%%%%%%%%%%%%%%%%%%%%%%%%%%%%%%%%%
\subsection{Finding the optimal bond from $\CF_3$}
%\paragraph{Finding the optimal bond from $\CF_3$.}
For each $k\in [1,n-2]$, we define the following quantities:
%\vspace{-.5cm}
\begingroup
\addtolength{\jot}{-1.2em}
\begin{align*}
&\pbr(k)=\max\{b_{j}+\sum_{l=1}^{j}a_l\ : \ j\in[1,k]\} \\
%\end{align*}
%\begin{align*}
&\pbri(k)=\min\{j\in[1,k] \ : \ \pbr(k)=b_{j}+\sum_{l=1}^{j}a_l\} \\
&\prk(k)=b_k+\sum_{l=1}^{k}a_l
\end{align*}
and for each $k\in [3,n]$, we define and then backwards calculate the following 
ones:
\begin{align*}
&\pbl(k)=\max\{b_{k-1}+\sum_{l=j}^{n-1}a_l + a_0\ : \ j\in[k,n]\} \hspace{2.6cm} \\
&\pbli(k)=\min\{j\in[k,n]  : \pbl(k)=b_{j-1}+\sum_{l=j}^{n-1}a_l + a_0\} \\
&\plk(k)=b_{k-1}+\sum_{l=k}^{n-1}a_l    \hspace{6.6cm}
\end{align*}
\endgroup
Similarly as before, these quantities can be computed in linear time using dynamic
programming by Procedure \textsc{Best-Bond-3}. 

\floatname{algorithm}{Procedure}
\begin{algorithm}[h]
\caption{\textsc{Best-Bond-3}}
\begin{algorithmic}[1]
    \State $\pbr(1) \gets b_{1}+a_1, \prk(1)\gets\pbr(1)$
    \State $\pbri(1) \gets 1$ 
    \For{$j=1,\ldots,n-3$}
        \State $\oc \gets \prk(j)+a_{j+1}-b_j+b_{j+1}$
        \If{$ \pbr(j)\geq \oc $}
                \State $\pbr(j+1) \gets \pbr(j)$
                \State $\pbri(j+1) \gets \pbri(j)$
        \Else
                \State $\pbr(j+1) \gets \oc$
                \State $\pbri(j+1) \gets j+1$
        \EndIf
	\State $\prk(j+1) \gets \prk(j)+a_{j+1}-b_j+b_{j+1}$
    \EndFor
    \State $\pbl(n) \gets b_{n-1}+a_0, \plk(n)\gets\pbl(n)$
    \State $\pbli(n) \gets n$ 
    \For{$j=n,\ldots,3$}
        \State $\oc \gets \plk(j)+a_{j-1}-b_{j-1}+b_{j-2}$
        \If{$ \pbl(j)\geq \oc $}
                \State $\pbl(j-1) \gets \pbl(j)$
                \State $\pbli(j-1) \gets \pbli(j)$
        \Else
                \State $\pbl(j-1) \gets \oc$
                \State $\pbli(j-1) \gets j-1$
        \EndIf
	\State $\plk(j-1) \gets \plk(j)+a_{j-1}-b_{j-1}+b_{j-2}$
    \EndFor
    \State $\bestsol \gets \pbr(2)+\pbl(2)$
    \For{$j=3,\ldots,n-1$}
	\State $\oc \gets \pbr(j)+\pbl(j)$
        \If{$ \bestsol < \oc $}
		\State $\bestsol \gets \oc$
        \EndIf
    \EndFor
    \State \textbf{return}$(\bestsol)$
\end{algorithmic}
\end{algorithm}

We observe two things: for every bond $S$ of type $3$, there exists a vertex $i\in\{2,\ldots,n-1\}$
such that $i$ does not belong to the path that is cut off by the bond $S$. If $S$ is the optimal
bond of type 3 and $i$ is the vertex not belonging to the path cut off by it, then the cost
of $S$ equals $\pbr(i)+\pbl(i)$. Thus, the cost of the optimal bond of type 3 is 
\[\max\left\{\pbr(i)+\pbl(i)  :  i\in \{2,\ldots,n-1\}\right \}\ .\]

To obtain the optimal bond, we compare the weights (costs) of the trivial bond and the 
optimal bonds of types 1, 2, and 3, and pick as our solution the best one. 
Note that the total running time is $\Oh{n}$.\qed
\end{proof}

Combining Theorem~\ref{thm:wheel} with the algorithm of Chimani et al.~\cite{CJN:23}
(cf. Theorem~\ref{thm:decomp} and Lemma~\ref{lem:2sum-minus}), 
yields the linear time algorithm for $(K_5\setminus e)$-minor-free graphs
(Corollary~\ref{cor:algorithm}).
For the sake of completeness, below we provide a complete description of the 
algorithm, building on the presentation of Chimani et al.~\cite{CJN:23}. 
Let $\MB(G)$ denote the size of the maximum bond in $G$, and given two vertices $u,v$
from $G$, let $\MB^{uv}(G)$ 
($\MB^u_v(G)$, resp.) denote the size of the 
maximum bond of $G$ in which the vertices $u,v$ are on the same side (on the
opposite sides, resp.) of the bond. 

We start by observing that given an algorithm for $\MB(G)$ running in time $p(|G|)$,
for every edge $uv\in E(G)$ we can construct $\MB^{uv}(G)$ in time $p(|G|)$, 
and the same holds for $\MB^u_v(G)$;
in the first case we let the algorithm construct $\MB(G')$ where $G'$ is the graph
obtained from $G$ by changing the weight of the edge $uv$ to $w(uv)=\sum_{e\in E}w(e)$,
and in the second case to $w(uv)= - \sum_{e\in E}w(e)$.

We proceed with a technical lemma. 
\begin{lemma}\label{lem:algorithm}
Let $G_1=(V_1,E_1)$ and $G_2=(V_2,E_2)$ be $2$-connected graphs such that $\{u,v\}= V_1\cap V_2$ and $\{uv\}= E_1\cap E_2$.
%, $uv\in E(G_1)\cap E(G_2)$ be an edge that appears in both of them.
Then
\begin{align*}
\MB(G_1\ksum{uv}^{-}G_2) & =\max\{\MB^{uv}(G_1\setminus uv), \MB(G_2')\} \\
\MB(G_1\ksum{uv}G_2)    & =\max\{\MB^{uv}(G_1), \MB(\bar G_2)\} 
\end{align*}
where $G_2'$ and $\bar G_2$, resp., is the graph
obtained from $G_2$ by changing the weight of the edge $uv$ to $w(uv)=\MB^u_v(G_1\setminus uv)$ and to $w(uv)=\MB^u_v(G_1)$, resp.
\end{lemma}
\begin{proof}
Let $G=G_1\ksum{uv}^{-}G_2$ and let $F$ be a maximum bond in $G$. By Lemma~\ref{lem:2sum-minus}, Cases~3 and 4,
    \begin{itemize}
        \item[a)] $F$ is a bond of $G_1\setminus uv$ with $u$ and $v$ on the same side, or
        \item[b)] $F$ is a bond of $G_2\setminus uv$ with $u$ and $v$ on the same side, or
        \item[c)] $F\cap E(G_i)$ is a bond of $G_i\setminus uv$ with $u$ and $v$ on different sides, for both $i\in\{1,2\}$.
       \end{itemize}
Thus, in case a), $\MB(G)=\MB^{uv}(G_1\setminus uv) \geq \MB(G_2')$. 
In case~b), $\MB(G)=\MB^{uv}(G_2\setminus uv) = \MB(G_2') \geq \MB^{uv}(G_1\setminus uv)$.
In case c), $\MB(G)=\MB^{u}_v(G_1\setminus uv) + \MB^{u}_v(G_2\setminus uv) = \MB(G_2')
 \geq \MB^{uv}(G_1\setminus uv)$.
In all three cases, we have the desired equality
 \begin{align*}
\MB(G_1\ksum{uv}^{-}G_2) & =\max\{\MB^{uv}(G_1\setminus uv), \MB(G_2')\} \ .
\end{align*}

If $G=G_1\ksum{uv}G_2$, we proceed in a similar way, using the fact 
%(cf.~\cite{Chaourar:20,CJN:23}) 
that by Lemma~\ref{lem:2sum-minus}, Case~2, $F$ is a bond in $G$ if and only if 
\begin{itemize}
        \item[a)] $F$ is a bond of $G_1$ with $u$ and $v$ on the same side, or
        \item[b)] $F$ is a bond of $G_2$ with $u$ and $v$ on the same side, or
        \item[c)] $F\cap E(G_i)$ is a bond of $G_i$ with $u$ and $v$ on different sides, for both $i\in\{1,2\}$. \qed
\end{itemize}
\end{proof}

\begin{corollary}\label{cor:algorithm}
The \MaxBond problem can be solved for any $(K_5\setminus e)$-minor-free graph in time $O(n)$.
\end{corollary}
\begin{proof}
First, we prove the claim for $2$-connected graphs. For a $2$-connected $(K_5\setminus e)$-minor-free graph $G=(V,E)$, 
by Theorem~\ref{thm:decomp}, we construct in linear time its decomposition
$G=G_1 \oplus^1 \dots \oplus^{l-1} G_\ell$ where each
$G_i$ is isomorphic to a wheel graph, $\prism$, $K_3$, or $K_{3,3}$,
and each operation $\oplus^i$ is $\oplus_2$ or $\oplus_2^-$.

By induction on $l$ we show the following: there exists a constant $c>0$ such that
given a decomposition of $G$ into
$G=G_1 \oplus^1 \dots \oplus^{l-1} G_\ell$ where each $G_i$ is isomorphic to 
a wheel graph, $\prism$, $K_3$, or $K_{3,3}$,
and each operation $\oplus^i$ is $\oplus_2$ or $\oplus_2^-$,
it is possible to compute $\MB(G)$ in time at most $2\cdot c\cdot \sum_{i=1}^l |V(G_i)|$. Since, $\sum_{i=1}^l |V(G_i)|=|V(G)|+2(\ell-1)$ and $\ell\leqslant |V(G)|$ we have $\sum_{i=1}^l |V(G_i)|\leqslant 3\cdot|V(G)|-2$ and hence the upper bound on the running time will follow.

If $l=1$, then $G$ is a wheel graph $W_n$, Prism, $K_3$ or
$K_{3,3}$; as each of them, except for $W_n$, is a constant size graph, and for 
the wheel graph $W_n$,
$\MB(W_n)$ can be computed in linear time by Theorem~\ref{thm:wheel}, 
we conclude, considering our initial observation of this section, that there
exists a constant $c>0$ such that for any $G$ of the graphs listed
in the previous sentence and any $uv\in E(G)$, 
computing $\MB^{uv}(G)$ and $\MB^u_v(G)$ takes at most time $c\cdot |V(G)|$.
% both $\MB^{uv}(G)$ and $\MB^u_v(G)$ can be computed in time at most $c\cdot |V(G)|$.

If $\ell \geq 2$, let 
$H=G_2 \oplus^2 \dots \oplus^{l-1} G_\ell$. We distinguish two cases: 
$\oplus^1 = \oplus_{uv}^-$ and $\oplus^1 = \oplus_{uv}$. In the first case,
let $H'$ be the graph obtained from $H$ by changing the weight of the edge $uv$ to $w(uv)=\MB^u_v(G_1\setminus uv)$; note that $H'$ has the
same decomposition as $H$, they differ only in the weight of the edge $uv$. 
Thus, by the inductive assumption, we can compute $\MB(H')$ in time 
$2\cdot c \cdot{\sum_{i=2}^l |V(G_i)|}$, and $\MB^{uv}(G_1)$ and $\MB^u_v(G_1)$ in time
$c\cdot |V(G_1)|$.
By Lemma~\ref{lem:algorithm}, 
\[\MB(G)=\MB(G_1\ksum{uv}^{-}H)  =\max\{\MB^{uv}(G_1\setminus uv), \MB(H')\}\ ,\]
therefore we can compute $\MB(G)$ from $\MB^{uv}(G_1\setminus uv)$ and $\MB(H')$ in
time $\Oh{1}$. Note that the time to construct $H'$ given $H$, $uv$ and $\MB^{u}_v(G_1)$, 
is $\Oh{1}$. Thus, exploiting the inductive assumption, we can compute $\MB(G)$ in time
$c\cdot |V(G_1)| + 2\cdot c\cdot \sum_{i=2}^l |V(G_i)| + \Oh{1} \leq 2\cdot c\cdot \sum_{i=1}^l |V(G_i)|$ which
completes the proof of the inductive step in the first case.

If $\oplus^1 = \oplus_{uv}$, we proceed analogously, exploiting the other equality
of Lemma~\ref{lem:algorithm}.

Finally, if the graph is not $2$-connected, 
we compute in linear time a decomposition of $G$ into $2$-connected components~\cite{HopcroftT:73}, 
construct the maximum bond for each of them in linear time, 
and output the largest of them; the total running time
will be $\Oh{|V(G)|}+\sum_{H\in \CC}\Oh{|V(H)|}=\Oh{|V(G)|}$ where $\CC$ is the set of $2$-connected 
components of $G$.
\qed
\end{proof}

%%%%%%%%%%%%%%%%%%%%%%%%%%%%%%%%%%%%%%%%%%%%%%%%%%%%%%%%%%%%%%%%%%%%%%%%%%%%%%%%%%%%%%%%%%%%%
\section{Concluding Remarks}
%%%%%%%%%%%%%%%%%%%%%%%%%%%%%%%%%%%%%%%%%%%%%%%%%%%%%%%%%%%%%%%%%%%%%%%%%%%%%%%%%%%%%%%%%%%%%
Our main result concerning $k$-sums for $k=1,2$ can be used in a natural way to get explicit descriptions of the bond polytope of the resulting graph. Let $G=G_1\ksum{1}G_2.$ Then Lemma \ref{lem:direct-sum-facets} and Lemma \ref{lem:1sum-vertices} allow us to explicitly obtain the inequalities describing \bondp{G} since it is just a subdirect sum of the two polytopes \bondp{G_1} and \bondp{G_2}. 
%We include a complete description of these inequalities in the Appendix (Lemma \ref{lem:direct-sum-facets}). 
Unfortunately, the number of inequalities is not additive and this cannot be avoided unless one constructs extended formulations, as we do.

One can also construct the inequalities describing \bondp{G_1\ksum{2}G_2} first by constructing the extended formulation in Theorem \ref{thm:2sum_xc} and then projecting out the additional coordinates that were added. Since there is only a constant
number of extra coordinates that need to be projected out, this can be done in polynomial time. However, since projecting out a single coordinate may asymptotically square the number of inequalities, generally speaking one would have neither a linear size description nor a linear (in output size) time construction. We leave it as an open problem whether one can describe the bond polytope under $2$-sums explicitly.
\begin{open}
    Let $G=G_1\ksum{e}G_2$ be obtained by a $2$-sum of $G_1$ and $G_2.$ 
    What is the description of $\bondp{G}$ (possibly in terms of the inequalities describing $\bondp{G_1}$ and $\bondp{G_2}$)?
    %What is the description of $\bondp{G}$, in terms of the inequalities describing $\bondp{G_1}$ and $\bondp{G_2}$?
\end{open}

Finally, for the $2$-sum operation where the common edge is removed, we believe that the extension complexity of $\bondp{G_1\ksum{2}^-G_2}$ is not additive in the extension complexities of $\bondp{G_1}$ and $\bondp{G_2}.$ This is because in Lemma \ref{lem:2sum_removed_xc_cases_1_2} we need the bond polytopes not of the summand graphs but of their subgraphs. We leave this as another open problem.

\begin{open}
    Let $G=G_1\ksum{e}^{-}G_2$ be obtained by a $2$-sum of $G_1$ and $G_2$ and removing the common edge $e.$ What is the description of $\bondp{G}$ (possibly in terms of the inequalities describing $\bondp{G_1}$ and $\bondp{G_2}$)?
\end{open}

%%%%%%%%%%%%%%%%%%%%%%%%%%%%%%%%%%%%%%%%%%%%%%%%%%%%%%%%%%%%%%%%%%%%%%%%%%%%%%%%%%%%%%%%%%%%%
\section*{Acknowledgment}
%%%%%%%%%%%%%%%%%%%%%%%%%%%%%%%%%%%%%%%%%%%%%%%%%%%%%%%%%%%%%%%%%%%%%%%%%%%%%%%%%%%%%%%%%%%%%
We would like to thank the anonymous referee for pointing out the connection between bonds and circuits of co-graphic matroids and for several other useful suggestions. 

\bibliographystyle{abbrv}
\bibliography{bond}

\end{document}